\documentclass[10pt]{amsart}%
\usepackage{amsmath}
\usepackage{amsfonts}
\usepackage{amsthm}
\usepackage{amssymb}
\usepackage{graphicx}%
\newtheorem{theorem}{Theorem}
\newtheorem*{theorem*}{Theorem}
\newtheorem*{acknowledgement*}{Acknowledgement}
\newtheorem*{definition*}{Definition}

\newtheorem{corollary}[theorem]{Corollary}

\newtheorem{lemma}[theorem]{Lemma}

\newtheorem{proposition}[theorem]{Proposition}
\newtheorem{remark}[theorem]{Remark}

\newcommand{\RR}[1]{{\mathbb{R}}^{#1}}
\newcommand{\pd}[2]{\frac{\partial #1}{\partial#2}}

\newcommand{\pdt}[0]{\frac{\partial}{\partial t}}

\newcommand{\mc}[1]{{\mathcal{#1}}}
\newcommand{\gt}[0]{\tilde{g}}

\newcommand{\Gam}[0]{\Gamma}
\newcommand{\Gamt}[0]{\widetilde{\Gamma}}

\newcommand{\Rc}[0]{\operatorname{Rc}}
\newcommand{\Rm}[0]{\operatorname{Rm}}

\newcommand{\Rmt}[0]{\widetilde{\operatorname{Rm}}}
\newcommand{\Rct}[0]{\widetilde{\Rc}}

\newcommand{\dfn}[0]{\doteqdot}

\newcommand{\wt}[1]{\widetilde{#1}}
\newcommand{\Rt}[0]{\wt{R}}
\newcommand{\nabt}[0]{\wt{\nabla}}

\title[An energy approach to uniqueness for the Ricci flow]
{An energy approach to the problem of uniqueness for the Ricci flow}

\author{Brett Kotschwar}
\address{Arizona State University, Tempe, AZ, USA}
\email{kotschwar@asu.edu}

\thanks{The author was supported in part by NSF grant DMS-0805834/DMS-1160613.}

\date{May 2012}

\begin{document}

\begin{abstract}
  We revisit the problem of uniqueness for the Ricci flow
  and give a short, direct proof, based on the consideration of a simple energy quantity,
  of Hamilton/Chen-Zhu's theorem on the uniqueness of complete solutions of uniformly bounded curvature.
  With a variation of this quantity and technique, we further prove a uniqueness theorem 
  for subsolutions to a general class of mixed differential inequalities which implies
  an extension of Chen-Zhu's result to solutions (and initial data) of
  potentially unbounded curvature.  
\end{abstract}
\maketitle


\maketitle

Let $M = M^n$ be a smooth manifold and $g_0$ a Riemannian metric on $M$. In this paper, we revisit the question of uniqueness
of solutions to the
initial value problem 
\begin{equation}\label{eq:rf}
    \pdt g(t) = -2\Rc(g(t)),\quad
         g(0) = g_0,
\end{equation}
associated to the Ricci flow on $M$. The broadest category in which uniqueness is currently 
known to hold without dimensional restrictions is that of complete solutions of uniformly bounded curvature. 
\begin{theorem}[Hamilton \cite{Hamilton3D}; Chen-Zhu \cite{ChenZhu}]\label{thm:rfuniqueness}
Suppose $g_0$ is a complete metric and $g(t)$ and $\gt(t)$ are solutions to the initial value problem \eqref{eq:rf} 
satisfying
  \[
\sup_{M\times[0, T]}|\Rm|_{g(t)},\; \sup_{M\times[0, T]}|\Rmt|_{\gt(t)} \leq K_0. 
\]
Then $g(t) = \gt(t)$ for all $t\in [0, T]$.
\end{theorem}

The uniqueness of solutions to \eqref{eq:rf} is not an automatic consequence of the theory of parabolic equations,
since the Ricci flow equation is only weakly-parabolic.  For compact $M$, there are two basic arguments, both due to 
Hamilton.  The first appears in Hamilton's orginal paper \cite{Hamilton3D} as 
a byproduct of the proof of the short-time existence of solutions
and is based on a Nash-Moser-type inverse function theorem.  The second, given in \cite{HamiltonSingularities}, effectively reduces the question 
of uniqueness to that for the strictly parabolic Ricci-DeTurck flow.  The basis of this argument is
the observation that the DeTurck diffeomorphisms, which are generally obtained as solutions to a system of ODE depending 
on a given solution to the Ricci-DeTurck flow, 
can also be represented as the solutions to a certain parabolic PDE
-- a harmonic map heat flow -- which depends on the associated solution to the Ricci flow.
As DeTurck's method is applicable to many other geometric evolution equations with gauge-based degeneracies, 
this second argument of Hamilton's gives rise to an elegant and flexible general prescription in which one exchanges the problem of uniqueness 
for one weakly parabolic system for the (separate) problems of existence and uniqueness for one or more auxiliary strictly parabolic systems. 

While Hamilton's prescription is also the basic template for Chen-Zhu's proof \cite{ChenZhu}
in the complete noncompact case, its components are not
straightforward to assemble in this setting. Given the lack
of general theory for the harmonic map flow into arbitrary target manifolds,
the authors in particular had to overcome the problem of 
short-time existence for their specific variant of this flow and, effectively, produce solutions 
(together with the crucial estimates) from scratch. This they accomplished
with the combination of a clever conformal transformation  of the initial metric and a series 
of intricate a priori estimates-- their approach producing, as independently useful byproducts, 
well-controlled solutions to both the harmonic map and Ricci-DeTurck flows associated to a given solution 
to the Ricci flow.

In this paper, however, we demonstrate that, if one is interested solely in the uniqueness of solutions to the Ricci flow,
it is possible to eliminate the passage through the harmonic map and Ricci-DeTurck-flows
and thus circumvent these delicate issues of existence entirely.
The idea, given two solutions $g(t)$ and $\gt(t)$ to the initial value problem \eqref{eq:rf}, is to consider
a quantity of the form
\begin{equation*}
 \mathcal{E}(t) = \int_M\left(t^{-\alpha}|g-\gt|^2_{g(t)}+ t^{-\beta}|\Gamma -\Gamt|^2_{g(t)} 
+ |\Rm-\Rmt|^2_{g(t)}\right)\Phi\; d\mu_{g(t)} 
\end{equation*}
for a suitable choice of the constants $\alpha$ and $\beta$ and weight function $\Phi= \Phi(x, t)$,
and to argue from the differential inequality it satisfies that it must vanish identically.  
This functional $\mc{E}$ is, in a sense, a compromise between 
the two perhaps most ``obvious'' candidates for such a quantity: the 
$L^2$-norms of the differences $g-\gt$ and $\Rm-\Rmt$.  Although neither $g-\gt$ nor $\Rm-\Rmt$ themselves 
satisfy a parabolic equation, they fail to do so in such a way
that their evolution equations, together with that of $\Gam-\Gamt$, can still be organized in a closed
and virtually parabolic system of inequalities to which the energy method may be applied.   

 When $M$ is compact, for example, it is not hard to show
that with $\alpha = \beta = 0$, and $\Psi\equiv 1$, we have 
$\mc{E}^{\prime}(t)\leq C\mc{E}$, and hence, since $\mc{E}(0) =0$, that $\mc{E}\equiv 0$.
When $M$ is non-compact, and the curvature of $g(t)$ and $\gt(t)$
is assumed to be uniformly bounded, we can take
$\alpha = 1$, $\beta \in (0, 1)$ and $\Phi$ of some sufficiently rapid decay in space in order
to draw much the same conclusion. With a similar argument and somewhat more careful estimation,
we can obtain the following rather inexpensive extension of Theorem \ref{thm:rfuniqueness}, which says, essentially,
that uniqueness holds in the class of solutions whose curvature at times $t > 0$ has at most quadratic
growth in the initial distance. 
Below $r_0(x) \dfn \operatorname{dist}_{g_0}(x, x_0)$  is the distance with respect to the metric $g_0$ from some fixed $x_0\in M$.

\begin{theorem}\label{thm:quadgrowth}
 Suppose $(M, g_0)$ is a complete noncompact Riemannian manifold satisfying the volume growth condition
\begin{equation}\label{eq:volgrowth}
    \operatorname{vol}_{g_0}(B_{g_{0}}(x_0, r)) \leq V_0e^{V_0r^2}
\end{equation}
for some constant $V_0$ and all $r > 0$ .  If $g(t)$ and $\gt(t)$ are smooth solutions
to \eqref{eq:rf} on $M\times [0, T]$ for which 
\begin{equation}
 \label{eq:unifequiv}
     \gamma^{-1} g_0(x) \leq g(x, t),\, \gt(x, t) \leq \gamma g_0(x),
 \end{equation}
and 
\begin{equation}\label{eq:quadgrowth}
    |\Rm(x, t)|_{g(t)} + |\Rmt(x, t)|_{\gt(t)} \leq \frac{K_0}{t^{\delta}}(r_0^2(x)+1),
\end{equation}
on $M\times (0, T]$ for some constants $\gamma$, $K_0$, and $\delta \in (0, 1/2)$,  then $g(t) = \gt(t)$ for all $t\in [0, T]$.
 \end{theorem}
In particular, we do not impose any condition on the curvature tensor of the initial metric,
although the volume condition \eqref{eq:volgrowth} is implied, for example,
by a bound of the form $\Rc(g_0)\geq -C(r_0^2+1)g_0$. Note that the uniform equivalence \eqref{eq:unifequiv}
is automatic in the case that the curvature tensors
of $g(t)$ and $\gt(t)$ are assumed bounded on the time-slices $M\times\{t\}$ for $t > 0$ (but blow-up
at a rate no greater than $t^{-\delta}$ as $t\searrow 0$), and so we have the following consequence.
\begin{corollary}
 If $g_0$ is a complete metric satisfying $\eqref{eq:volgrowth}$ and $g(t)$ and $\gt(t)$ are smooth solutions to \eqref{eq:rf} satisfying
\[
      |\Rm(x, t)|_{g(t)}+ |\Rmt(x, t)|_{\gt(t)} \leq \frac{K_0}{t^{\delta}},
\]
on $M\times (0, T]$ for some $K_0$ and $\delta\in (0, 1/2)$, then $g(t) = \gt(t)$ for all $t\in [0, T]$.
\end{corollary}
With Shi's existence theorem \cite{Shi}, we also have have the following special case.
\begin{corollary} If $g_0$ is complete and of bounded curvature, then any solution $g(t)$
to the initial value problem \eqref{eq:rf} which remains uniformly equivalent to $g_0$ and obeys the bound 
$|\Rm(x, t)| \leq K_0(r_0^2(x) + 1)$
on $M\times (0, T]$ must have bounded curvature tensor.
\end{corollary}

We note that certain extensions of Theorem 1 and other related topics can be found in the papers
 \cite{Chen}, \cite{ChenYan}, \cite{Fan}, \cite{GiesenTopping}, \cite{Hsu}, \cite{LuTian}, \cite{Topping}. 
For example, in \cite{Chen}, Chen  proves that if $(M^3, g_0)$ is complete, has bounded nonnegative sectional
curvature, and satisfies a certain uniform volume condition, then any two complete solutions to the Ricci flow
with this initial data must agree identically; for surfaces, he proves the same result for initial data
with Gaussian curvature of arbitrary sign.  It is an interesting question whether his result may be extended
to higher dimensions and to initial metrics with some curvature growth. 
Since we have only used rather classical estimates in this paper
and have made no fine use of the nonlinear reaction terms in the evolution
of the curvature, we expect that the combination of conditions \eqref{eq:unifequiv}
and \eqref{eq:quadgrowth} in Theorem \ref{thm:quadgrowth} can be relaxed further.
However, the general technique we describe may perhaps nevertheless
be of use to such future efforts as a strategy to bypass the gauge-degeneracy of the equation and possibly also
of use for the quantitative comparison of solutions that ``agree'' in some limiting sense as 
$t\to T\in[-\infty, \infty]$.

\section{Preliminaries}

Going forward, we assume that $g(t)$ and $\gt(t)$ are complete solutions to \eqref{eq:rf}.
We will select one of the metrics, $g(t)$, as our reference metric and use the notation
$V\ast W$ below to represent a linear combination of contractions of the tensors $V$ and $W$ with respect 
to the metric $g(t)$ in which the coefficients and the total number of terms are bounded by some constant
depending only on the dimension.  In the estimates, we will use $C = C(n)$ to such denote such a constant,
which may change from line to line. To further reduce clutter in our expressions we will often use
$|\cdot|\dfn |\cdot|_{g(t)}$ to denote the norms induced on $T^k_l(M)$ by $g(t)$ (distinguishing other norms
by a subscript), and simply write $R =\Rm$ and $\Rt = \Rmt$ for the $(3, 1)$ curvature tensors associated with $g$ and $\gt$.
(We will use $\operatorname{scal}(g)$ for the scalar curvature.) Finally, we will use the notation $T^{*} \dfn \max\{T, 1\}$.
\subsection{Evolution equations and inequalities}

Define
\[
 h\dfn g - \gt, \quad A \dfn \nabla -\nabt,\quad S \doteqdot R -\Rt,
\]
that is,  $A^{k}_{ij} \dfn \Gamma_{ij}^k-\widetilde{\Gamma}_{ij}^k$, and $S_{ijk}^l \dfn R_{ijk}^l - \Rt_{ijk}^l$.  
We begin by organizing the evolution equations satisfied by these tensors in such a way that every term contains either a factor of 
(some contraction of) one of the group or $\nabla S$.
We note for later
that we can write
\begin{equation}\label{eq:ahidentities}
  g^{ij}- \gt^{ij} = -g^{ik}\gt^{jl}h_{kl}, \quad\mbox{and}\quad \nabla_k \gt^{ij} = \gt^{aj}A_{ka}^i + \gt^{ia}A^j_{ka},
\end{equation}
or, according to our convention,
\[
g^{-1}-\gt^{-1} = \gt^{-1}\ast h,\quad\mbox{and}\quad \quad\nabla \gt^{-1} = \gt^{-1}\ast A.
\]

Now, on $M\times [0, T]$, the tensors $h$ and $A$ satisfy 
\[
\pdt h_{ij} = -2(R_{ij} - \tilde{R}_{ij}) = -2S^l_{lij}
\]
and
\begin{align}\label{eq:aevol}
\begin{split}
\pdt A_{ij}^k &= \gt^{mk}\left(\nabt_{i}\tilde{R}_{jm} + \nabt_{j}\tilde{R}_{im}- \nabt_{m}\tilde{R}_{ij}\right)\\
  &\phantom{=}-g^{mk}\left(\nabla_{i}R_{jm} + \nabla_{j}R_{im}- \nabla_mR_{ij}\right),
\end{split}
\end{align} respectively.  Since
\begin{equation}\label{eq:diff}
  \nabt_{i}\tilde{R}_{jk} - \nabla_{i}\tilde{R}_{jk} = A_{ij}^p\tilde{R}_{pk} + A_{ik}^p\tilde{R}_{jp}, 
\end{equation}
and $\nabla_{i}(R_{jk} - \Rt_{jk}) = \nabla_{i}S^{l}_{ljk}$, we can use the first equation in \eqref{eq:ahidentities} to put the equation
for $\pdt A$ in the schematic form
\begin{align}\label{eq:aevolschem}
   \pdt A &=\gt^{-1}\ast h\ast\nabt \Rt + A\ast\tilde{R}  +C \nabla S,
\end{align}
using  to obtain the last term.
Similarly, using
\begin{align}\label{eq:revol}
\begin{split}
	\pdt R^{l}_{ijk} &= \Delta R^{l}_{ijk} 
	+ g^{pq}\left( R^r_{ijp}R^l_{rqk} - 2R^{r}_{pik}R^l_{jqr} 
	+ 2 R^l_{pir} R^r_{jqk}\right)\\
		&\phantom{=\Delta R^{l}_{ijk} } 
	- g^{pq}\left(R_{ip} R^l_{qjk} + R_{jp} R^l_{iqk} + R_{kp} R_{ijq}^l\right)  
	+ g^{pl}R_{pq} R^q_{ijk},
\end{split}
\end{align}
together with the generalization $(\nabla - \nabt)W = A\ast W$ of \eqref{eq:diff} to arbitrary tensors $W$, 
we can express the evolution equation of $S$ in the schematic form
\begin{align}\begin{split}\label{eq:sevolschem}
  \pdt S &= \nabla_a(g^{ab}\nabla_b R - \gt^{ab}\nabt_b\tilde{R}) + \gt^{-1}\ast A \ast \nabt\tilde{R}\\
	&\phantom{=}+ \gt^{-1}\ast h \ast \Rt\ast\Rt + S\ast R + S\ast \Rt.
\end{split}
\end{align}
Here the shorthand $\nabla_a(g^{ab}\nabla_b R - \gt^{ab}\nabt_b\tilde{R})$ represents
$\nabla_{a}\left(g^{ab}\nabla_bR^{l}_{ijk} - \gt^{ab}\nabt_b\Rt^{l}_{ijk}\right)$, and we have
obtained \eqref{eq:sevolschem} from \eqref{eq:ahidentities} and \eqref{eq:revol} via the following explicit representation of the $\gt(t)$-Laplacian of $\Rt$ as
\begin{align*}
 \widetilde{\Delta}\Rt_{ijk}^l &= \gt^{ab}\nabt_a\nabt_b\Rt_{ijk}^l = \nabt_{a}\left(\gt^{ab}\nabt_b\Rt_{ijk}^l\right)\\
\begin{split}
 &=\nabla_a\left(\gt^{ab}\nabt_b\Rt_{ijk}^l\right) -A_{ap}^a\gt^{pb}\nabt_b\Rt_{ijk}^l -A_{ap}^l\gt^{ab}\nabt_b\Rt_{ijk}^p\\	
 &\phantom{=}\quad+ A_{ai}^p\gt^{ab}\nabt_b\Rt_{pjk}^l + A_{aj}^p\gt^{ab}\nabt_b\Rt_{ipk}^l + A_{ak}^p\gt^{ab}\nabt_b\Rt_{ijp}^l.
\end{split}
\end{align*}

Now using the Cauchy-Schwarz inequality together with the above representations, we obtain that
\begin{equation}\label{eq:hanorm}
 \left|\pdt h\right| \leq C|S|,\quad \left|\pdt A\right| \leq C\left(|\gt^{-1}||\nabt\tilde{R}||h| + |\tilde{R}||A| + |\nabla S|\right),
\end{equation}
and
\begin{align}\begin{split}\label{eq:snorm}
  &\left|\pd{S}{t} - \Delta S -  \operatorname{div} U\right| \\
	&\qquad\qquad\leq C\left(|\gt^{-1}||\nabt\Rt||A| + |\gt^{-1}||\Rt|^2|h| + (|R|+|\Rt|)|S|\right).
\end{split}
\end{align}
where $U \dfn g^{ab}\nabla_b\Rt -\gt^{ab}\nabt_b\Rt$ is the section of $T^2_3(M)$ given in local coordinates by
\begin{align}
\begin{split}\label{eq:udef}
 U^{al}_{ijk} &= g^{ab}\nabla_b\Rt^l_{ijk} - \gt^{ab}\nabt_b\Rt_{ijk}^l\\
  &= (g^{ab}-\gt^{ab})\nabt_b\Rt^l_{ijk} 
      + g^{ab}(\nabla_b\Rt_{ijk}^l - \nabt_b\Rt_{ijk}^l)\\
  &= -g^{ak}\gt^{bl}h_{kl}\nabt_b\Rt_{ijk}^l + g^{ab}\left(A_{bp}^l\Rt_{ijk}^p - A_{bi}^p\Rt_{pjk}^l
    - A_{bj}^p\Rt_{ipk}^l - A_{bk}^p\Rt_{ijp}^l\right),
\end{split}
\end{align}
and by $\operatorname{div} U$ we mean the section of $T_{3}^1(M)$ given
by $(\operatorname{div} U)_{ijk}^l = \nabla_{a}U^{al}_{ijk}$. Observe that $U$ satisfies
\begin{equation}\label{eq:unorm}
    |U| \leq C(|\gt^{-1}||\nabt\Rt||h| + |A||\Rt|).
\end{equation}

We leave these evolution inequalities in the above rather raw form for the moment and 
return to simplify them later in forms specialized
to the specific assumptions of Theorems \ref{thm:rfuniqueness} and \ref{thm:quadgrowth}.
\subsection{A function of rapid decay}
In order that our energy quantity $\mc{E}$ be well-defined and that the differentiations and integrations-by-parts we will ultimately 
wish to perform be valid, we will need to integrate against a weight function of sufficiently rapid decay. 
\begin{lemma}\label{lem:cutoffgrowth}
  Suppose $\bar{g}(t)$ is a smooth family of complete metrics on $M\times [0, T]$ satisfying $\gamma^{-1} \bar{g} \leq g(t)$,
  where $\bar{g} = \bar{g}(0)$ and $x_0\in M$. Define $\bar{r}(x) = \operatorname{dist}_{\bar{g}(0)}(x, x_0)$. 
  Then, for any positive constants $L_1$ and $L_2$, there exists a positive constant $T^{\prime} = T^{\prime}(n, \gamma, L_1, L_2, T)$,
  and a function $\eta: M\times [0, T^{\prime}]\to \RR{}$ that is smooth in $t$,
  Lipschitz (and smooth $d\mu_{g(t)}-a.e.$) on each $M\times \{t\}$, and that simultaneously satisfies
  the conditions
\[
-\pd{\eta}{t} + L_1 |\eta|^2_{\bar{g}(t)} \leq 0,\quad\mbox{and}\quad
  e^{-\eta} \leq e^{-L_2\bar{r}^2(x)},
\]
 on $M\times[0, \tau]$ whenever $0 < \tau \leq T^{\prime}$.
\end{lemma}
\begin{proof} We follow the construction in Chapter 12 of \cite{RFV2P2},
(cf. also \cite{KarpLi}, \cite{LiYau}), and define $\eta(x, t) \dfn \eta_{B, \tau}(x, t) \dfn B\bar{r}^2(x)/(4(2\tau-t))$
for on $M\times[0, \tau]$. The function $\bar{r}$ is continuous and smooth off of the $\bar{g}$-cut locus of $x_0$, where it satisfies $|\nabla \bar{r}|^2_{\bar{g}(0)}= 1$,
and hence $|\nabla\bar{r}|^2_{g(t)}\leq \gamma$. It follows that
that $\bar{r}$ is Lipschitz and smooth $d\mu_{\bar{g}(t)}$-a~.e~.
for each $t\in [0, \tau]$. For each $L_1 > 0$, we have
\[
 -\pd{\eta}{t} + L_1 |\nabla\eta|_{\bar{g}(t)}^2 \leq -B(1- BL_1\gamma)\frac{\bar{r}^2(x)}{4(2\tau-t)^2},
\]
and we can guarantee the first condition provided $B < 1/(\gamma L_1)$.  Also, for any $t\in [0, \tau]$, we have
\[
  B\bar{r}^2(x)/(8\tau^2) \leq \eta(x, t) \leq  B\bar{r}^2(x)/(4\tau^2),
\]
so, given $L_2 > 0$,  we can ensure
the second condition on $M\times [0, \tau]$ provided $0 < \tau \leq T^{\prime} \dfn \min\{(B/(8L_2))^{1/2}, T\}$ 
\end{proof}
 
\section{The case of uniformly bounded curvature}

Although Theorem \ref{thm:quadgrowth} is strictly stronger than Theorem \ref{thm:rfuniqueness},
the proof can be substantially simplified in the case of bounded curvature, and thus we give a separate
argument here to demonstrate the technique. We will consider only the case of non-compact $M$,
as the proof for that of compact $M$ is nearly identical, but less involved, and can be easily reconstructed from the argument
below. For the remainder of this section, we will assume that $g_0$, $g(t)$, and $\gt(t)$ satisfy the assumptions of
Theorem \ref{thm:rfuniqueness}, and that $\beta\in (0, 1)$ is a fixed constant.

\subsection{Derivative and decay estimates}
We first recall that the global
estimates of Bando \cite{Bando} and Shi \cite{Shi} imply that there exists a constant $N = N(n, K_0, T^*)$
such that
\begin{align}
  \begin{split}\label{eq:globalshi}
      |R|_{g(t)}+ |\Rt|_{g(t)} + \sqrt{t}|\nabla R|_{g(t)} 
    + \sqrt{t}|\nabt\Rt|_{\gt(t)} + t|\nabla\nabla R|_{g(t)} + t|\nabt\nabt\Rt|_{\gt(t)}\leq N.
  \end{split}
\end{align}
on $M\times [0, T]$
It is a standard argument (cf., e.g., Theorem 14.1 in \cite{Hamilton3D}) that the uniform curvature bounds
on $g(t)$ and $\gt(t)$ imply that the metrics $g(t)$, $\gt(t)$, and $g_0$ remain uniformly equivalent.
Thus, the estimates above hold (for some potentially larger $N$) when the norms are replaced by any one of
$|\cdot|\dfn|\cdot|_{g(t)}$, $|\cdot|_{\gt(t)}$, and $|\cdot|_{g_0}$.  In what follows, we will use $N$ to denote
a series of constants depending only on $n$, $K_0$, and $T^*$ 
which may vary from one inequality to the next.

We begin by noting that the same argument that yields the uniform equivalence of the metrics can be used to produce
simple estimates on the decay of $h$ and $A$ as $t \searrow 0$ that imply, in particular, that
 $t^{-1}|h|^2$ and and $t^{-\beta}|A|^2$ tend to zero uniformly as $t\searrow 0$ on $M$ and can be  continuously extended to $M\times [0,T]$.
\begin{lemma}\label{lem:hadecay} Under the assumptions of Theorem \ref{thm:rfuniqueness}, we have
$|h(p, t)| \leq Nt$ and $|A(p, t)| \leq N\sqrt{t}$
on $M\times [0, T]$ for some constant $N = N(n, K_0, T^{*})$.
\end{lemma}
\begin{proof}
  At an arbitrary $p$ in $M$, we have
\begin{align*}
	|h(p, t)| &\leq N|h(p, t)|_{g_0} \leq N\int_0^t|S(p, s)|_{g_0}\,ds\\
	    &\leq  N\int_0^t\left(|R(p, s)|_{g_0} + |\Rt(p, s)|_{g_0}\right)\,ds\leq Nt,
\end{align*}
and similarly, using \eqref{eq:aevol}, for any $0 < \epsilon < t$, we have 
\begin{align*}
     |A(p, t) - A(p, \epsilon)| &\leq N |A(p, t) - A(p, \epsilon)|_{g_0}
	      \leq N\int_{\epsilon}^t\left(|\nabla R(p, s)|_{g_0} + |\nabt \Rt(p, s)|_{g_0}\right)\,ds\\
      &\leq N\int_{\epsilon}^t s^{-1/2}\,ds \leq N(\sqrt{t} - \sqrt{\epsilon}).
\end{align*}
Sending $\epsilon\to 0$ completes the proof.
\end{proof}

\subsection{Definition and differentiability of $\mc{E}$}\label{ssec:diffE}
Next, since the uniform curvature bound on $g(0)$ implies a lower bound on $\Rc(g(0))$, 
the Bishop-Gromov volume comparison
theorem implies that 
\[
\operatorname{vol}_{g(0)}(B_{g(0)}(x_0, r)) \leq Ne^{Nr}
\]
for some constant $N$ and all $r > 0$. Since the metrics $g(t)$ and $g_0$ are uniformly equivalent, we thus also have
\[
\operatorname{vol}_{g(t)}(B_{g_0}(x_0, r)) \leq Ne^{Nr}
\]
for some $N$. With any choice of $B >0$ (and independent of $T$), 
the function $\eta =\eta_{B, T}$ of Lemma \ref{lem:cutoffgrowth}
(with $\bar{g}(t) = g(t)$) will satisfy $e^{-\eta}\leq e^{-Br_0^2/(8T)}$ on $M\times [0, T]$, so if we choose $B > 0$ sufficiently
small to ensure, say, that
\[
      \pd{\eta}{t} - 3|\nabla\eta|^2 \geq 0,
\]
it will still follow that any continuous function of at most (sub-quadratic) exponential growth in $r_0(x)$ at $t$ will be $e^{-\eta}\,d\mu$-integrable.
So $\pd{\eta}{t}$ and $|\nabla\eta|^2$, being of quadratic growth in $r_0(x)$, are $e^{-\eta}\,d\mu$-integrable for any $t\in [0, T]$
(where, here and elsewhere, we write $d\mu\dfn d\mu_{g(t)}$),
as are the uniformly bounded quantities $t^{-1}|h|^2$, $t^{-\beta}|A|^2$, and $|S|^2$. 
Moreover, since 
\[
\nabla\Rt = \nabla\Rt + A\ast\Rt,\quad\mbox{and}\quad \nabla\nabt \Rt = \nabt\nabt\Rt + A\ast\nabt\Rt,
\]
it follows from equations \eqref{eq:aevolschem}, \eqref{eq:sevolschem}, \eqref{eq:hanorm}, and \eqref{eq:globalshi} that $|\nabla S|$ and 
$|\nabla_a(g^{ab}\nabla_bR - \gt^{ab}\nabt_b\Rt)|$, and hence $\pdt|h|^2$, $\pdt|A|^2$, and $\pdt|S|^2$,
 are uniformly bounded on $M\times[\epsilon, T]$ for any $\epsilon > 0$,
and consequently $e^{-\eta}\,d\mu$-integrable for each $t\in (0, T]$ (although they need not be bounded as $t\to 0$).

Since $\pdt d\mu = -\operatorname{scal}(g(t))\,d\mu$, and the scalar curvature of $g(t)$ is bounded by assumption,
these observations imply that, for fixed $\beta\in (0, 1)$ and $\eta$ as above, the quantity
\[
  \mc{E}(t) \dfn \int_{M}\left(t^{-1}|h|^2 + t^{-\beta}|A|^2 + |S|^2\right)e^{-\eta}\, d\mu
\]
is differentiable on $(0, T]$, and (with the dominated convergence theorem) satisfies $\lim_{t\searrow 0} \mc{E}(t) = 0$.

\subsection{Vanishing of $\mc{E}$}

We claim in fact that $\mc{E}$ vanishes identically on $[0, T]$. This is a consequence of iterating the following result. 

\begin{proposition} There exists $N_0 = N_0(n, K_0, T^*) > 0$ and $T_0 = T_0(n, \beta) \in (0, T]$ such that
$\mc{E}^{\prime}(t) \leq N_0\mc{E}(t)$ for all $t\in (0, T_0]$.  Hence $\mc{E}\equiv 0$ on $(0, T_0]$. 
\end{proposition}
\begin{proof} For $t\in (0, T]$ and $\alpha \in (0, 1)$,  define
\begin{align*}
  \mc{G}(t) &\dfn \int_M|S|^2e^{-\eta}\,d\mu, \quad \mc{H}(t) \dfn t^{-1}\int_{M}|h|^2e^{-\eta}\,d\mu,\\
  \mc{I}(t) &\dfn t^{-\beta}\int_{M}|A|^2e^{-\eta}\,d\mu,
  \quad\mbox{and}\quad \mc{J}(t) \dfn \int_M|\nabla S|^2e^{-\eta}\,d\mu,
\end{align*}
so $\mc{E}(t) = \mc{G}(t) + \mc{H}(t) + \mc{I}(t)$.
In view of the discussion in Section \ref{ssec:diffE}, we may freely differentiate under the integral sign and integrate by parts
about any $0 < t \leq T$.  As before, in the estimates below, $C$ will denote a series of constants depending only on $n$,
and $N$ a series of constants depending at most on $n$, $\beta$, $K_0$, and $T^*$.
Taking into account the time dependency of the norms $|\cdot| = |\cdot|_{g(t)}$
and the measure $d\mu = d\mu_{g(t)}$, and using \eqref{eq:snorm} together with \eqref{eq:globalshi} we have
\begin{align*}
      \mc{G}^{\prime} &\leq N \mc{G} + \int_M \left(2\left\langle \pd{S}{t}, S\right\rangle - \pd{\eta}{t}|S|^2\right)e^{-\eta}\,d\mu\\
      \begin{split}
			 &\leq N \mc{G} + \int_M\bigg(2\left\langle \Delta S + \operatorname{div} U, S\right\rangle+ C|\gt^{-1}||\nabt\Rt||A||S| 
				    + C|\gt^{-1}||\Rt|^2|h||S|\\
	  &\phantom{\leq N \mc{G}\int_M\bigg(2\langle \Delta S} + (|R|+|\Rt|)|S|^2 -\pd{\eta}{t}|S|^2\bigg)e^{-\eta}\,d\mu
      \end{split}\\	     
      \begin{split}
			 &\leq N \mc{G} + \int_M\left(2\left\langle \Delta S + \operatorname{div} U, S\right\rangle + Nt^{-1/2}|A||S|+ N|h||S|
					 -\pd{\eta}{t}|S|^2\right)e^{-\eta}\,d\mu
      \end{split}\\
        \begin{split}
 			 &\leq N \mc{G} + t\mc{H} + t^{\beta - 1}\mc{I} +
		  \int_M\left(2\left\langle \Delta S + \operatorname{div} U, S\right\rangle -\pd{\eta}{t}|S|^2\right)e^{-\eta}\,d\mu,
       \end{split}	  
\end{align*}
for $t > 0$, where we have estimated 
\[
    Nt^{-1/2}|A||S| \leq t^{\beta-1}(t^{-\beta}|A|^2) + N|S|^2 \quad\mbox{and}\quad N|h||S| \leq t(t^{-1})|h|^2 + N|S|^2,
\]
to obtain the third line.
We can then integrate by parts in the last integral to obtain
 \begin{align*}
 \begin{split} 	 
 	     &\int_M\left(2\left\langle \Delta S + \operatorname{div}{U}, S\right\rangle -\pd{\eta}{t}|S|^2\right)e^{-\eta}\, d\mu\\
 	     &\qquad\leq - 2\mc{J} + \int_M\left( 2|\nabla\eta||\nabla S||S| + 2|\nabla S||U| + 2|\nabla \eta||U||S| -\pd{\eta}{t}|S|^2\right)e^{-\eta}\, d\mu\\
 	     &\qquad\leq -\mc{J}   + \int_M\left(\left(3|\nabla\eta|^2 -\pd{\eta}{t}\right)|S|^2 + 3|U|^2\right)e^{-\eta}\,d\mu.      
 \end{split}
 \end{align*}
Here we have estimated
\[
  2|\nabla\eta||\nabla S||S|+ 2|\nabla S||U|\leq |\nabla S|^2 + 2|\nabla\eta|^2|S|^2 +2|U|^2,
\]
and
\[ 
  2|\nabla \eta||U||S| \leq |\nabla \eta|^2|S|^2 + |U|^2.
\]
Since $|U|^2 \leq Nt^{-1}|h|^2 + N|A|^2$ by \eqref{eq:unorm} and \eqref{eq:globalshi}, and $\pd{\eta}{t} \geq 3|\nabla\eta|^2$ by assumption,
putting everything together, we have
\begin{align}
\nonumber
  \mc{G}^{\prime} &\leq N \mc{G} + (t + N)\mc{H} + (t^{\beta - 1} + Nt^{\beta})\mc{I}  - \mc{J}\\
 		 &\leq N\mc{G} + N\mc{H} + (t^{\beta -1}+ N)\mc{I} - \mc{J},	 \label{eq:gprime}	  
 \end{align}
for any $t \in (0, T]$, using $t\leq T^*$, and $t^{\beta} \leq (T^*)^{\beta}$.	

Similarly, with \eqref{eq:hanorm} and \eqref{eq:globalshi}, we compute that
\begin{align}\nonumber
  \mc{H}^{\prime} &\leq (N -t^{-1})\mc{H} + t^{-1}\int_M\left( 2\left\langle \pd{h}{t}, h\right\rangle -\pd{\eta}{t}\right)e^{-\eta}\,d\mu\\
      \nonumber   &\leq (N -t^{-1})\mc{H} + \int_M Ct^{-1}|S||h|e^{-\eta}\,d\mu\\
		    &\leq  (N -(1/2)t^{-1})\mc{H}  + C\mc{G} \label{eq:h1prime}
\end{align}
where we have used that $\pd{\eta}{t} \geq 0$ and $Ct^{-1}|S||h|\leq (1/2)t^{-2}|h|^2 + C|S|^2$

Finally, and using \eqref{eq:hanorm} and \eqref{eq:globalshi} again, we 
\begin{align}\nonumber
  \mc{I}^{\prime} &\leq (N-\beta t^{-1})\mc{I} + t^{-\beta}\int_M\left( 2\left\langle \pd{A}{t}, A\right\rangle -\pd{\eta}{t}|A|^2\right)e^{-\eta}\,d\mu\\
  \nonumber &\leq \leq (N-\beta t^{-1})\mc{I} + t^{-\beta}\int_MC\left(|\gt^{-1}||\nabt\Rt||h| + |\Rt||A| + |\nabla S|\right)|A|e^{-\eta}\,d\mu\\
\nonumber		    &\leq (N-\beta t^{-1})\mc{I} + \int_M\left( Nt^{-1/2-\beta}|h||A| + Ct^{-\beta}|\nabla S||A|\right)e^{-\eta}\,d\mu\\
\label{eq:h2prime}		    &\leq N\mc{H} + (N-\beta t^{-1} + Ct^{-\beta})\mc{I} + \mc{J}
\end{align}
where we have again used that $\pd\eta{t} \geq 0$ and have estimated
\[
       Nt^{-1/2-\beta}|h||A| + Ct^{-\beta}|\nabla S||A| \leq Nt^{-1}|h|^2 + Ct^{-2\beta}|A|^2 + |\nabla S|^2.
\]

Combining \eqref{eq:gprime}, \eqref{eq:h1prime}, and \eqref{eq:h2prime},
we obtain that, for any $t\in (0, T]$,
\[
      \mc{E}^{\prime}(t) \leq N\mc{E}(t) -(1/2)t^{-1}\mc{H}(t) - t^{-1}(\beta - t^{\beta} + Ct^{1-\beta})\mc{I}(t).			     
\]
Thus, for $T_0$ sufficiently small depending only on $\beta$ and $C = C(n)$, and for some $N_0 = N_0(n, K_0, \beta, T^*)$ sufficiently large, 
we have $\mc{E}^{\prime}(t) \leq N_0 \mc{E}(t)$ on $(0, T_0]$.
Since $\lim_{t\searrow 0} \mc{E}(t) = 0$, it follows from Gronwall's inequality that $\mc{E}\equiv 0$ on $[0, T_0]$.
\end{proof}

\section{The case of potentially unbounded curvature}

In this section, we reduce Theorem \ref{thm:quadgrowth} to a special case of a general result, Theorem \ref{thm:ext},
in the following section. The strategy is essentially the same as that of the bounded curvature setting,
but here we will need to organize our estimates more carefully in order to
`squeeze' the differential inequality satisfied by our energy quantity sufficiently to absorb the growth of coefficients
that we were able to regard as uniformly bounded in our previous computations.
\subsection{Derivative estimates and consequences}
First, we recall the following refined form of
Shi's local first derivative estimate (due to Hamilton, \cite{HamiltonSingularities})
in which the dependencies of the bound on the local curvature bound, the radius of the ball,
and the elapsed time are explicit.

\begin{theorem}[Shi/Hamilton] \label{thm:shi} Suppose that $(M, g_0)$ is an open Riemannian manifold in which, for a given $x_0\in M$ and $r > 0$, 
the closure of $B_{g_0}(x_0, r)$ is compactly contained. Then, if $g(t)$ is a solution to \eqref{eq:rf} on $M\times [0, T]$ with 
\begin{equation*}
    \sup_{B_{g_0}(x_0, r)\times[0, T]}|R|_{g(t)} \leq K_0,
\end{equation*}
there exists a constant $C = C(n)$ such that
\begin{equation*}
 |\nabla R|_{g(t)} \leq CK_0\left(\frac{1}{r^2} + \frac{1}{t} + K_0\right)^{1/2},
\end{equation*}
on $B_{g_0}(x_0, r/2)\times (0, T]$.
\end{theorem}
We can use this result on a series of domains to obtain a derivative bound for solutions satisfying assumptions
of the general form of those in Theorem \ref{thm:quadgrowth}.

\begin{corollary} Suppose $(M, g_0)$ is a complete noncompact Riemannian manifold. If $g(t)$ is a smooth family of complete metrics solving \eqref{eq:rf}
and satisfying
\[
  t^{\delta}|R|_{g(t)}(x, t) \leq K_0(r_0^2(x)+1)
\]
on $M\times [0, T]$ for some constant $K$ and $\delta \in [0, 1]$, where $r_0(x) = \operatorname{dist}_{g_0}(x, x_0)$ for some $x_0 \in M$,
then there exists a constant $K = K(n, \delta, K_0, T^*)$
such that
\begin{equation}\label{eq:localderbound}
   t^{\delta + 1/2}|\nabla R|_{g(t)}(x,t ) \leq K (r_0^2(x)+1)^{3/2}.
\end{equation}
\end{corollary}
\begin{proof}
  Let $t_0 \in (0, T]$ and $\epsilon \dfn t_0/2$. For any $r > 0$, we have 
\[
      |R(x, t)|\leq \epsilon^{-\delta}K_0(r^2 + 1)
\]
on $B_{g_0}(x_0, r) \times [\epsilon, T]$, and so by Theorem \ref{thm:shi},
there is a $C = C(n)$ such that 
\[
   t^{\delta+1/2}|\nabla R|(x, t) \leq C\left(\frac{t}{\epsilon}\right)^{\delta}K_0(r^2 + 1)\left(\frac{t}{r^2} + \frac{t}{t-\epsilon} + \frac{t}{\epsilon^{\delta}}
		K_0(r^2 + 1)\right)^{1/2},
\]
on $B_{g_0}(x_0, r/2) \times (\epsilon, T]$.  Thus, for any $r \geq 1$, we have
\begin{align*}
 t_0^{\delta + 1/2}|\nabla R|(x, t_0) &\leq C2^{\delta}K_0(r^2 + 1)
\left(t_0 + 2 + t_0^{1-\delta}2^{\delta}K_0(r^2 + 1)\right)^{1/2}\\
  &\leq K^{\prime}(r^2+1)^{3/2},
\end{align*}
for some $K^{\prime} = K^{\prime}(n, \delta, K_0, T^*)$ on $B_{g_0}(x_0, r/2)$.
Since $t_0\in (0, T]$ and $r\geq 1$ were arbitrary, this implies that
\[
  \sup_{B_{g_0}(x_0, r)\times [0, T]} t^{\delta + 1/2}|\nabla R|(x, t)\leq 8K^{\prime}(r^2 + 1)^{3/2}
\]
for any $r \geq 1$.

 For the pointwise estimate, we note that if $x\in B_{g_0}(x_0, 1)$, we have
\[
  \frac{t^{\delta + 1/2}}{(r_0^2(x) + 1)^{3/2}}|\nabla R|(x, t) \leq  t^{\delta + 1/2}|\nabla R|(x, t) \leq 16K^{\prime},
\]
and if $x\in M\setminus B_{g_0}(x_0, 1)$, we have
\begin{align*}
 \frac{t^{\delta + 1/2}}{(r_0^2(x) + 1)^{3/2}}|\nabla R|(x, t) &\leq  \frac{t^{\delta + 1/2}}{(r_0^2(x) + 1)^{3/2}}
   \left\{\sup_{y\in B_{g_0}(x_0, 2r_0(x))}|\nabla R|(y, t)\right\}\\  
 &\leq 8K^{\prime}\left(\frac{4r_0^2(x) + 1}{r_0^2(x) + 1}\right)^{3/2} \leq 64K^{\prime}.
\end{align*}
This verifies \eqref{eq:localderbound} with $K = 64K^{\prime}$.
\end{proof} 

In Theorem \ref{thm:quadgrowth}, we assume that $g(t)$ and $\gt(t)$ are uniformly equivalent
to $g_0$, and so, effectively, that both $|R(x, t)|_{g_0}$ and $|\Rt(x, t)|_{g_0}$ have
at most quadratic growth in $r_0(x)$. With an argument exactly analogous to Lemma \ref{lem:hadecay},
and using \eqref{eq:localderbound},
we could then obtain global bounds of the form $|h(x, t)|\leq Nt^{1-\delta}(r_0^2(x)+1)$
and $|A(x, t)|\leq Nt^{1/2-\delta}(r_0^2(x)+1)^{3/2}$ on the decay of $h$ and $A$ as $t\searrow 0$. 
In fact, since the proof of Theorem \ref{thm:ext} below is based on localized energy quantities,
we will not need global estimates on either $h$ or $A$, and instead we just note that (since $g(t)$
and $\gt(t)$ are assumed to be smooth solutions which agree at $t =0$) we have naive estimates of the form $|h(x, t)|\leq Pt$
and $|A(x, t)|\leq Pt$ on $\Omega\times [0, T]$ for some $P = P(\Omega, g, \gt)$ and
any compact $\Omega \subset M$. 

\begin{lemma}\label{lem:hadecay2} For any $r > 0$, 
there exists a constant $P$ depending on $\gamma$ and the maximum values of
$|R|_{g_0}$, $|\Rt|_{g_0}$, $|\nabla R|_{g_0}$
and $|\nabt\Rt|_{g_0}$ on $\overline{B_{g_0}(x_0, r)}\times [0, T]$ such that
\begin{equation}\label{eq:adecay2}
 \sup_{x\in B_{g_0}(x_0, r)} \left(|h(x, t)| + |A(x, t)|\right)\leq Pt.
\end{equation}
\end{lemma}
\begin{proof}
Just define
\[
    \tilde{P}(r) \dfn \sup_{\overline{B_{g_0}(x_0, r)}\times [0, T]}\left(|R|_{g_0} + |\Rt|_{g_0}+ |g^{-1}|_{g_0}|\nabla\Rc|_{g_0} + |\gt^{-1}|_{g_0}|\nabt\Rct|_{g_0}\right).
\]
Then, by \eqref{eq:aevol}, for any $r > 0$ and $x \in B_{g_0}(x_0, r)$, we have $|\pd{h}{t}|_{g_0} \leq 2\sqrt{n}\tilde{P}(r)$ and $|\pd{A}{t}|_{g_0} \leq 3\tilde{P}(r)$
and $A(x, 0) = 0$, so  $|h(x, t)|\leq \gamma|h(x, t)|_{g_{0}}\leq 2\sqrt{n}\gamma\tilde{P}t$, and 
$|A(x, t)| \leq \gamma^{3/2}|A(x, t)|_{g_0} \leq 3\gamma^{3/2}\tilde{P}t$ on $B_{g_0}(x_0, r)$.
\end{proof}

\subsection{Evolution inequalities for $h$, $A$, and $S$ revisited}

We next organize the inequalities satisfied by the time-derivatives of
 $t^{-\alpha}|h|^2$, $t^{-\beta}|A|^2$, and $|S|^2$ so that the coefficients
$|R|$, $|\Rt|$, $|\nabla R|$ and $|\nabt\Rt|$ are distributed across the totality of terms
in a way that their growth can be adequately absorbed. 
Going forward, we will write 
\[
\rho(x) \dfn r_0^2(x) + 1.
\]
\begin{lemma}\label{lem:rawineq} For any solutions $g(t)$, $\gt(t)$ to the Ricci flow on $M\times [0, T]$ and any 
$\alpha$, $\beta$, $\delta$, $\sigma \in \RR{}$,
there exists a constant $C = C(n)$ such that, on $M\times (0, T]$,
\begin{align}\label{eq:hraw}
 &\pdt\left(\frac{\rho^2}{t^{\alpha}}|h|^2\right) \leq \frac{C\rho^2}{t^{\alpha}}\left(|R| +\frac{\rho}{t^{\sigma}} - \frac{\alpha}{t}\right)|h|^2
	      + \frac{C\rho}{t^{\alpha - \sigma}}|S|^2\\
\begin{split}\label{eq:araw}
  &\pdt\left(\frac{\rho}{t^{\beta}}|A|^2\right) \leq \frac{C\rho}{t^\beta}\left(\left(|\gt^{-1}||\Rt| + |R|\right) 
	  +\frac{\rho}{t^{\delta}} |\gt^{-1}|^2 + \frac{\rho}{t^{\beta}} - \frac{\beta}{t}\right)|A|^2\\
	  &\phantom{\pdt\left(\frac{\rho}{t^{\beta}}|A|^2\right) \leq} + \frac{C}{t^{\beta-\delta}}|\nabt\Rt|^2|h|^2+ \frac{1}{4}|\nabla S|^2
\end{split}\\
\begin{split}\label{eq:sraw}
  &\left\langle \pd{S}{t} -\Delta S - \operatorname{div} U, S\right\rangle \leq C|\gt^{-1}|^2|\nabt \Rt|^{4/3}|A|^2
		  + C|\gt^{-1}|^2|\Rt|^3|h|^2\\
  &\phantom{\left\langle \pd{S}{t} -\Delta S - \operatorname{div}_{g(t)}U, S\right\rangle \leq} 
		+ C\left(|\nabt\Rt|^{2/3} + \left(|R|+|\Rt|\right)\right)|S|^2 
\end{split}
\end{align}
and
\begin{equation}\label{eq:uraw}
      |U|^2 \leq C\left(|\gt^{-1}|^2|\nabt\Rt|^2|h|^2 + |\Rt|^2|A|^2\right),
\end{equation}
where $U$ is as in \eqref{eq:udef}.
\end{lemma}
\begin{proof} 
Since $\pdt|h|^2 \leq C|R||h|^2 + 2\langle\pd{h}{t}, h\rangle$, we find that
\begin{align*}
 \pdt\left(\frac{\rho^2}{t^{\alpha}}|h|^2\right) &\leq \frac{C\rho^2}{t^{\alpha}}\left(|R| - \frac{\alpha}{t}\right)|h|^2
		+\frac{2\rho^2}{t^{\alpha}}\left\langle\pd{h}{t}, h\right\rangle\\
	&\leq \frac{C\rho^2}{t^{\alpha}}\left(|R| - \frac{\alpha}{t}\right)|h|^2 + \frac{C\rho^2}{t^{\alpha}}|S||h|\\
	&\leq \frac{C\rho^2}{t^{\alpha}}\left(|R| +\frac{\rho}{t^{\sigma}} - \frac{\alpha}{t}\right)|h|^2 + \frac{C\rho}{t^{\alpha - \sigma}}|S|^2
\end{align*}
where we have estimated
\[
	  \frac{\rho^2}{t^{\alpha}}|S||h| \leq \frac{C\rho}{t^{\alpha}}\left(\frac{\rho^2}{t^{\sigma}}|h|^2 + t^{\sigma}|S|^2\right).
\]
Likewise, from \eqref{eq:hanorm}, we have
\begin{align*}
 \pdt\left(\frac{\rho}{t^{\beta}}|A|^2\right) &\leq \frac{C\rho}{t^{\beta}}\left(|R| - \frac{\beta}{t}\right)|A|^2
	+ |A|\frac{C\rho}{t^{\beta}}\left(|\gt^{-1}||\nabt\Rt||h| +|\Rt||A| + |\nabla S|\right),\\
\begin{split}
	&\leq \frac{C\rho}{t^{\beta}}\left(\left(|R|+ |\Rt|\right) + \frac{\rho}{t^{\delta}}|\gt^{-1}|^2 + \frac{\rho}{t^{\beta}} - \frac{\beta}{t}\right)|A|^2\\
	&\phantom{\leq} + \frac{C}{t^{\beta -\delta}}|\nabt\Rt|^2|h|^2 + \frac{1}{4}|\nabla S|^2,
\end{split}
\end{align*}
where we have used
\[
    \frac{\rho}{t^{\beta}}|\gt^{-1}||\nabt\Rt||h||A| 
      \leq \frac{C}{t^{\beta}}\left(\frac{\rho^2}{t^{\delta}}|\gt^{-1}|^2|A|^2 + t^{\delta}|\nabt\Rt|^2|h|^2\right),
\]
and
\[
      \frac{C\rho}{t^{\beta}}|A||\nabla S| \leq \frac{C\rho^2}{t^{2\beta}}|A|^2 + \frac{1}{4}|\nabla S|^2.
\]
Next, from \eqref{eq:snorm}, we have
\begin{align*}
 &\left\langle \pd{S}{t} -\Delta S - \operatorname{div} U, S\right\rangle \leq C|S|\left(|\gt^{-1}||\nabt\Rt||A|
	  + |\gt^{-1}||\Rt|^2|h| + \left(|R|+|\Rt|\right)|S|\right)\\
	&\qquad\qquad\leq C\left(|R|+ |\Rt| + |\nabt\Rt|^{2/3}\right)|S|^2 
	    + C|\gt^{-1}|^2|\nabt\Rt|^{4/3}|A|^2 + C|\Rt|^3|h|^2,
\end{align*}
since
\[
  |\gt^{-1}||\nabt\Rt||A||S| \leq C\left(|\gt^{-1}|^2|\nabt\Rt|^{4/3}|A|^2 + |\nabt\Rt|^{2/3}|S|^2\right),
\]
and
\[
    |\gt^{-1}||\Rt|^2|h||S|\leq C\left(|\gt^{-1}|^2|\Rt|^3|h|^2 + |\Rt||S|^2\right).
\]
Finally, from \eqref{eq:unorm}, we have $|U|^2 \leq C\left(|g^{-1}|^2|\nabt\Rt|^2|h|^2 + |\Rt|^2|A|^2\right)$.
\end{proof}

We now specialize to the setting of Theorem \ref{thm:quadgrowth} (except that we continue to permit 
$\delta \in [0, 1]$) and define
\begin{equation}\label{eq:absdef}
    \alpha \dfn (3+\delta)/2, \quad \beta \dfn \alpha - 1 =  (1+\delta)/2,\quad \sigma \dfn \alpha/2 = (3+\delta)/4,
\end{equation}
(so that, in particular, $\alpha \leq 2$ and $\beta$, $\sigma \leq 1$ if $\delta\in [0, 1]$).
Together with the derivative estimate
\eqref{eq:localderbound}, we can put the above inequalities into the following simplified form.
\begin{proposition}\label{prop:refineq}
 Suppose $g_0$ is a complete metric satisfying the volume growth condition \eqref{eq:volgrowth}, 
$g(t)$ and $\gt(t)$ are solutions to \eqref{eq:rf} satisfying \eqref{eq:unifequiv} and the curvature bound \eqref{eq:quadgrowth}
for some $\delta \in [0, 1]$, and $\alpha$, $\beta$, $\sigma$ are as in \eqref{eq:absdef}.
 Then there exists a constant $N = N(n, \gamma, \delta, K, T^{*})$ 
 such that, on $M\times (0, T]$,
\begin{align}\label{eq:hmod}
 \pdt|\bar{h}|^2 &\leq \frac{N\rho}{t^{\sigma}}\left(|\bar{h}|^2
	      + |S|^2\right),\\
\label{eq:amod}
\pdt|\bar{A}|^2 &\leq \frac{N\rho}{t^{\sigma}}\left(|\bar{h}|^2+ |\bar{A}|^2 \right) +  \frac{1}{4}|\nabla S|^2,\\
\label{eq:smod}
  \left\langle \pd{S}{t} -\Delta S - \operatorname{div} U, S\right\rangle &\leq
	\frac{N\rho}{t^{\sigma}}\left(|\bar{h}|^2 + |\bar{A}|^2 
		+ |S|^2\right), 
\end{align}
and
\begin{align}\label{eq:umod}
|U|^2 \leq \frac{N\rho}{t^{\sigma}}\left(|\bar{h}|^2 + |\bar{A}|^2\right),
\end{align}
where $\bar{h} \dfn \rho t^{-\alpha/2}h$, $\bar{A} \dfn \rho^{1/2}t^{-\beta/2}A$,
and $\rho(x) = r_0^2(x) + 1$ are as before.
\end{proposition}
\begin{proof}
By assumption, we have bounds on $|R|$ and $|\Rt|_{\gt}$, and
from \eqref{eq:localderbound}, estimates on $|\nabla R|$ and $|\nabt\Rt|_{\gt}$. Since
$g(t)$ and $\gt(t)$ are uniformly equivalent, we also have estimates on $|\Rt|$ and $|\nabt\Rt|$, and
we collect all of these estimates here with a common constant $K_1 = K_1(n, \delta, \gamma, K_0, T^*)$:
\[
   \sup_{M\times[0, T]}t^{\delta}\left(|R| + |\Rt|\right) \leq K_1\rho,
  \quad \sup_{M\times[0,T]} t^{\delta+1/2}\left(|\nabla R| + |\nabt\Rt|\right) \leq K_1\rho^{3/2}.
\]
The assumption of uniform equivalence also implies 
\[
|\gt^{-1}| \leq  \gamma |\gt^{-1}|_{g_0} \leq \gamma^2|\gt^{-1}|_{\gt} = \gamma^2\sqrt{n}.
\]
So now it is just a matter of substituting these bounds into \eqref{eq:hraw}, \eqref{eq:araw}, \eqref{eq:sraw}, and \eqref{eq:uraw}.
In what follows we will use $N$ to denote any constant that depends only 
on $n$, $\delta$, $\gamma$,  $K_1$, and $T^{*}$.

First, from \eqref{eq:hraw}, using $\sigma = (3+\delta)/4$, we have
\begin{align}
\nonumber \pdt|\bar{h}|^2 &\leq 
	  \left(|R| +\frac{\rho}{t^{(3+\delta)/4}} - \frac{\alpha}{t}\right)|\bar{h}|^2
	      + \frac{\rho}{t^{\alpha-(3+\delta)/4}}|S|^2\\
\nonumber &\leq \frac{\rho}{t^{(3+\delta)/4}}\left(K_1t^{3(1-\delta)/4}+ 1\right)|\bar{h}|^2 +
	      \frac{C\rho}{t^{(3+\delta)/4}}|S|^2\\
\label{eq:hmod1} &\leq \frac{N\rho}{t^{(3+\delta)/4}}(|\bar{h}|^2 + |S|^2)
\end{align}
for any $t > 0$, since $\alpha =(3+\delta)/2$ and $t^{3(1-\delta)/4} \leq T^{*}$.
Next, from \eqref{eq:araw}, we have
 \begin{align}
 \begin{split}
 \nonumber
   \pdt |\bar{A}|^2 &\leq C\left(\left(|\gt^{-1}||\Rt| + |R|\right) 
 	  +\frac{\rho}{t^{\delta}}|\gt^{-1}|^2 + \frac{\rho}{t^{\beta}} - \frac{\beta}{t}\right)|\bar{A}|^2\\
 	  &\phantom{\leq}+ \frac{Ct^{\alpha - \beta+ \delta}}{\rho^2}|\nabt\Rt|^2|\bar{h}|^2
		    + \frac{1}{4}|\nabla S|^2
 \end{split}\\
\nonumber
 &\leq \frac{N\rho}{t^{(1+\delta)/2}}\left(1 + t^{(1-\delta)/2}\right)|\bar{A}|^2+ \frac{N\rho}{t^{\delta}}|\bar{h}|^2
		    + \frac{1}{4}|\nabla S|^2\\
\label{eq:amod1}
&\leq \frac{N\rho}{t^{(1+\delta)/2}}\left(|\bar{h}|^2 + |\bar{A}|^2\right) + \frac{1}{4}|\nabla S|^2,
 \end{align}
since $\beta - \delta = (1-\delta)/2 \geq 0$, and $\alpha -\beta + \delta - (1 + 2\delta) = -\delta$.
Similarly,
\begin{align}
\begin{split}
\nonumber
  &\left\langle \pd{S}{t} -\Delta S - \operatorname{div} U, S\right\rangle\\
  &\;\;\leq C\bigg\{|\gt^{-1}|^2|\nabt \Rt|^{4/3}|A|^2 + |\gt^{-1}|^2|\Rt|^3|h|^2 +\left(|\nabt\Rt|^{2/3} + \left(|R|+|\Rt|\right)\right)|S|^2\bigg\} 
\end{split}\\
\nonumber
\begin{split}
  &\;\;\leq \frac{N\rho}{t^{(2+4\delta)/3-\beta}}|\bar{A}|^2
		  + \frac{N\rho}{t^{3\delta - \alpha}}|\bar{h}|^2
		+ N\rho\left(\frac{1}{t^{(1+2\delta)/3}} + \frac{1}{t^{\delta}}\right)|S|^2 
\end{split}\\
  \nonumber
\begin{split}
  &\;\;\leq \frac{N\rho}{t^{(1+5\delta)/6}}|\bar{A}|^2
		  + \frac{N\rho}{t^{(5\delta - 3)/2}}|\bar{h}|^2
		+ \frac{N\rho}{t^{(1+2\delta)/3}}\left(1+ t^{(1-\delta)/3}\right)|S|^2 
\end{split}\\
\label{eq:smod1}
  &\;\;\leq \frac{N\rho}{t^{(1+2\delta)/3}}\left(|\bar h|^2 + |\bar A|^2 + |S|^2\right)
\end{align}
since $(2+4\delta)/3 - \beta = (1+5\delta)/6$, $3\delta -\alpha = (5\delta -3)/2$,
and
\[
  \frac{5\delta -3}{2} \leq \frac{1+ 5\delta}{6} \leq \frac{1+2\delta}{3},
\]
when $\delta\in [0, 1]$.

Finally, from \eqref{eq:uraw}, we have
\begin{align}
\nonumber
      |U|^2 &\leq C\left(|\gt^{-1}|^2|\nabt\Rt|^2|h|^2 + |\Rt|^2|A|^2\right)
\leq N\rho\left(\frac{1}{t^{1+2\delta - \alpha}}|\bar{h}|^2 + \frac{1}{t^{2\delta - \beta}}|\bar{A}|^2\right)\\
\label{eq:umod1}	  
       &\leq \frac{N\rho}{t^{(3\delta -1)/2}}\left(|\bar{h}|^2 + |\bar{A}|^2\right),
\end{align}
since $1+2\delta - \alpha = 2\delta - \beta = (3\delta -1)/2$.  Since 
$\sigma = (3+\delta)/4$ is the largest exponent of $1/t$ to appear in the coefficients
of 
\eqref{eq:hmod1}, \eqref{eq:amod1}, \eqref{eq:smod1}, and \eqref{eq:umod1}, we obtain
\eqref{eq:hmod}, \eqref{eq:amod}, \eqref{eq:smod}, and \eqref{eq:umod}. 
\end{proof}

\subsection{A remark on the general strategy}

At this point, we could multiply $|\bar h|^2$, $|\bar{A}|^2$, and $|S|^2$ by suitable cutoff and
decay functions and introduce localized versions of the integral quantities 
$\mc{G}$, $\mc{H}$, $\mc{I}$, and $\mc{J}$ from the last section in order to argue as before. However,
the argument we will use is only dependent on the structure of the system of inequalities
in Proposition \ref{prop:refineq} and this structure is not really specific to the Ricci flow.

The guiding principle behind our efforts thus far has been that, 
while $g$ and $\gt$ themselves do not satisfy strictly parabolic
equations, the individual curvature tensors $R$ and $\Rt$ do, and do so with respect to elliptic operators
$\Delta$ and $\widetilde{\Delta}$ whose coefficients, respectively, depend on $g$ and $\gt$ and their derivatives (up to second-order).
Since the difference of $g$ and $\gt$ (and those of their derivatives) can be controlled in an essentially ordinary-differential way
by the difference of $R$ and $\Rt$ and those of their derivatives,  
by prolonging the system for $S = R - \Rt$ to include just as many of the differences of $g$ and $\gt$
and their derivatives as are needed to control $\Delta - \widetilde{\Delta}$, we can hope to obtain a closed
and nearly parabolic system of inequalities.  As we have 
seen (thanks to an integration by parts) we need only to add $h$ and $A$ to obtain
a system for which uniqueness can be established much as for strictly parabolic systems. 

Thus the strategy is, first, to treat the curvature tensors (rather than the metrics)
as the central objects in the problem, second, to prolong the system 
by lower order quantities whose vanishing is somewhat logically redundant for the purposes of establishing uniqueness,
but whose inclusion allows us to account for the lack of a common elliptic operator in the separate
parabolic equations satisfied by the curvature tensors, and, third, to attempt to apply the energy method
to the mixed (but, in practice, nearly parabolic) system of inequalities satisfied by the aggregate of the curvature
and lower-order quantities. This strategy can be used to encode uniqueness problems for other geometric evolution
equations, such as the mean curvature flow, into those for similar systems of inequalities,
and thus we formulate a somewhat more general uniqueness result than is necessary to prove Theorem \ref{thm:quadgrowth}
in the chance that it might be of some independent interest.

This approach to handling gauge-degeneracies in evolution equations involving curvature, 
is similar to that employed in \cite{KotschwarBackwards}, \cite{KotschwarHolonomy}, and has its origins in the work of Alexakis \cite{Alexakis}
on unique-continuation for the vacuum Einstein equations. (See also \cite{WongYu}.)

\subsection{Reduction to system of mixed differential inequalities}

For the application of the result of the next section to the situation of  Theorem \ref{thm:quadgrowth},
we'll take (in the notation of that section)
 $\mc{X} = T^1_{3}(M)$, $\mc{Y} = T_{2}^0(M)\oplus T_{2}^1(M)$ and $X(t) = S(t)\in \mc{X}$,
$Y(t) = \bar{h}(t)\oplus \bar{A}(t)\in \mc{Y}$, and take one of the metrics, $g(t)$, as our family of reference metrics.
Note that, by assumption, $X$ is smooth on $M\times [0, T]$ and  satisfies $X(x, 0) \equiv 0$ and
\[
  |X(x,t)|^2\leq K_1t^{-2\delta}(r_0^2(x)+1)^2 \leq Nt^{-\sigma^{\prime}}e^{Nr_0^2(x)} 
\]
on $M\times (0, T]$ for an appropriate constant $N$ where 
\[
\sigma^{\prime} \dfn \operatorname{max}\{(3+\delta)/4, 2\delta\} < 1.
\]
(It is here that we need $\delta < 1/2$, as opposed to $\delta < 1$.)
The family of sections $Y(t)$ is smooth in $t$ and Lipschitz over $M$ (smooth but for the factors of $\rho$) for $t > 0$, 
and, as noted in Lemma \ref{lem:hadecay2}, satisfies
\[
      |Y(x, t)|^2 = t^{-\alpha}\rho^2(x)|h(x, t)|^2 + t^{-\beta}\rho|A(x, t)|^2 \leq t^{2-\alpha}(r^2+1)^2P(r)
\]
on $B_{g_0}(x_0, r)\times[0, T]$ for any $r > 0$. Thus $|Y(x, t)|$ tends to zero uniformly as $t\searrow 0$ on any compact set. (We will not need 
any assumptions on the behavior of $Y(t)$ at spatial infinity.)  In terms of $X$ and $Y$ (and the norms on $\mc{X}$ and $\mc{Y}$
induced by $g(t)$), Proposition \ref{prop:refineq} implies
\begin{align*}
	  \left\langle \pd{X}{t} - \Delta X - \operatorname{div} U, X \right\rangle & \leq \frac{N\rho}{t^{\sigma^{\prime}}}(|X|^2 + |Y|^2),\\
	  \left\langle \pd{Y}{t}, Y\right\rangle &\leq + \frac{1}{4}|\nabla X|^2 + \frac{N\rho}{t^{\sigma^{\prime}}}(|X|^2 + |Y|^2),
\end{align*}
on $M\times (0, T]$, with $|U|^2\leq (N\rho/t^{\sigma^{\prime}})(|X|^2 + |Y|^2)$.  Note that, although equations \eqref{eq:hmod} and \eqref{eq:amod}
 only directly imply an inequality on $\pdt|Y|^2$, in our derivation of these inequalities, we immediately estimated
the contribution of the time-derivative of $|\cdot|=|\cdot|_{g(t)}$ from above by terms proportional to 
$|R||Y|^2 \leq CK_1(\rho/t^{\delta})|Y|^2 \leq N(\rho/t^{\sigma^{\prime}})|Y|^2$,
so we in fact also have the inequalities in the  above weaker (although somewhat more symmetric) form.  Theorem \ref{thm:quadgrowth} now
follows at once from Theorem \ref{thm:ext} below.

\section{A uniqueness theorem for systems of virtually parabolic differential inequalities.}

Let $(M, g_0)$ be a complete Riemannian manifold,
satisfying
\begin{equation}\label{eq:vol}
 \operatorname{vol}(B_{g_0}(x_0, r)) \leq A_0e^{A_0r^2},
\end{equation}
for some $x_0\in M$, constant $A_0$, and all $r > 0$. Define $r_0(x) \dfn \operatorname{dist}_{g_0}(x_0, x)$ and $\rho(x)\dfn r_0^2(x)+1$
as before.  Suppose that $g(t)$ is a smooth family of metrics on $M\times [0, T]$ such that, writing $\pdt g_{ij} = -2p_{ij}$, 
the conditions
\begin{equation}\label{eq:gassumptions}
   \gamma^{-1} g_0 \leq g(t) \leq \gamma g_0, \quad\mbox{and}\quad t^{-\sigma}|p(x, t)| \leq N_0\rho(x)
\end{equation}
are satisfied
for some constants $\gamma$, $\sigma\in (0, 1)$, and $N_0$, where $|\cdot| \dfn |\cdot|_{g(t)}$ as before.

Now let $\mc{X} = \bigoplus_{i=1}^rT^{k_i}_{l_i}(M)$ and $\mc{Y} = \bigoplus_{i=1}^{r^{\prime}}T^{k^{\prime}_i}_{l^{\prime}_i}(M)$
represent tensor bundles over $M$ equipped with the metrics and connections induced by $g(t)$ and $\nabla(t)$, the Levi-Civita connection of $g(t)$. 
\begin{theorem}\label{thm:ext} For any choice of $a$, $\sigma \in (0, 1)$,  $\gamma > 0$, and nonnegative constants $A_0$, $A_1$, $N_0$, and $N_1$,
there exists $T_0 = T_0(n, \gamma,\sigma, a, A_0, A_1, N_0, N_1) > 0$, such that whenever $g(t)$ is a smooth family of metrics
on $M\times [0, T_0]$ satisfying \eqref{eq:vol} and \eqref{eq:gassumptions}, and $X(t) \in C^{\infty}(\mc{X})$, $Y(t)\in C(\mc{Y})$ are families of sections of
depending smoothly on $t\in (0, T_0]$, that satisfy
\begin{equation}
    \lim_{t\searrow 0}\sup_{x\in \Omega}|X(x, t)| = 0, \quad \lim_{t\searrow 0}\sup_{x\in \Omega}|Y(x, t)| =0
\end{equation}
on every compact $\Omega \subset M$, the growth bound
\begin{equation}\label{eq:xygrowth}
  t^{\sigma}|X(x, t)|^2 \leq A_1e^{A_1r_0^2(x)},
\end{equation}
on $M\times (0, T_0]$, and the system of inequalities
\begin{align}\label{eq:exteqxy}\begin{split}
    \left\langle\pd{X}{t} - \Delta X - \operatorname{div}U, X\right\rangle 
&\leq \frac{a}{2}|\nabla X|^2+ \frac{N_1\rho}{t^{\sigma}}\left(|X|^2 +  |Y|^2\right),\\  
    \left\langle\pd{Y}{t}, Y\right\rangle &\leq \frac{a}{2}|\nabla X|^2 + \frac{N_1\rho}{t^{\sigma}}\left(|X|^2 + |Y|^2\right),
\end{split}
\end{align}
on $M\times (0, T_0]$ for
 $U(t)\in C^{\infty}(TM\otimes\mc{X})$ satisfying
\begin{equation}\label{eq:uest}
|U|^2 \leq \frac{N_1\rho}{t^{\sigma}}\left(|X|^2 + |Y|^2\right),
\end{equation}
 then $X(t)\equiv 0$, $Y(t)\equiv 0$ for all $t\in (0, T_0]$. 
\end{theorem}
\begin{remark}
    Here, by $\operatorname{div} U = \operatorname{div}_{g(t)}U$ we mean the section of $\mc{X}$ whose value in the fiber over $(x, t) \in M\times (0, T_0]$ 
is $\sum_{i=1}^{n}\nabla_{e_i}U(e_i, \cdot)$ for a $g(t)$-orthonormal basis $\{e_i\}_{i=1}^n$
of $T_xM$.
\end{remark}
\begin{remark}
 With a simple modification of the proof below (and an appropriate reduction of the constant $a$ in the statement)
 one can substitute for the operator $\Delta = \Delta_{g(t)}$ in Theorem \ref{thm:ext}
 any elliptic operator of the form $\mc{L} = \Lambda^{ij}\nabla_i\nabla_j$ where $\Lambda(t)\in C^{\infty}(T^2_0{M})$ satisfies
 $\lambda^{-1}g^{ij}(x, t)\leq \Lambda^{ij}(x, t)\leq \lambda g^{ij}(x, t)$ on $M\times (0, T_0]$
 with 
  \[
  |\nabla \Lambda|^2\leq N\rho/t^{\sigma}
  \]
for some constants $\lambda$ and $N$.
\end{remark}

\begin{proof}  Again, we'll assume that $M$ is non-compact, as the argument in the compact case is very similar and less involved.
For the time being we will take $0<T_0\leq T$ to be a small constant to be determined later.
We begin by introducing a suitable cutoff function.  Choose a nonincreasing $\psi\in C^{\infty}(\RR{}, [0, 1])$ satisfying
\begin{equation*}
 \left\{\begin{array}{ll}
    \psi\equiv 1 &\mbox{on}\quad (-\infty, 1/2]\\
    \psi\equiv 0 &\mbox{on}\quad [1, \infty)
 \end{array}\right.
\end{equation*}
and $(\psi^{\prime})^2 \leq C\psi$.  The function $\phi_{r}: M\to [0, 1]$ defined by
  $\phi_{r}(x) = \psi(r_0(x)/r)$ then satisfies
\[
  \phi_{r}\equiv 1 \quad\mbox{on}\quad B_{g_0}(x_0, r/2),\quad \phi_{r}\equiv 0 
    \quad\mbox{on}\quad M\setminus B_{g_0}(x_0, r),
\]
and is Lipschitz (smooth off of the $g_0$-cut locus of $x_0$). On account of the uniform equivalence of $g(t)$
with $g_0$, we have
\[
      |\nabla\phi_{r}|^2 \leq C\gamma r^{-2}\phi_r 
\]
off of a $d\mu_{g_0}$- (hence $d\mu_{g(t)}$-) set of measure zero in $B_{g_0}(x_0, r) \times [0, T_0]$.

Then, for any $r > 0$ and $t > 0$, we define
\[
    \mc{G}_{r} \dfn \int_{M}|X|^2\phi_re^{-\eta}\,d\mu, \quad\mc{H}_{r} \dfn \int_{M}|Y|^2\phi_re^{-\eta}\,d\mu,\quad
 \mc{J}_{r} \dfn \int_{M}|\nabla X|^2 \phi_re^{-\eta}\,d\mu,  
\]
with $\mc{E}_r \dfn \mc{G}_r + \mc{H}_{r}$, where $\eta = \eta_{B, T_0}$ is as in Lemma \ref{lem:cutoffgrowth}
with $B = B(n, a, \gamma) > 0$ taken small enough to ensure
that
\begin{equation}\label{eq:etaeq}
  \pd{\eta}{t} - \frac{5-2a}{2(1-a)}|\nabla\eta|^2\geq 0
\end{equation}
on $M\times [0, T_0]$. As noted in that lemma, this can be achieved independently of our choice of $T_0$,
and is not affected by a further reduction of $T_0$.
Below we will continue to use $C = C(n)$ to denote a universal constant and $N$ any constant which depends at most on $n$, the ranks $(k_i, l_i)$,
$(k^{\prime}_i, l^{\prime}_i)$ from the definitions of $\mc{X}$ and $\mc{Y}$, and the constants $a$, $\gamma$, $\sigma$, $A_0$, $A_1$, 
$N_0$, and $N_1$. For convenience,
we'll write $\theta  \dfn \theta(x, t)\dfn \rho(x)/t^{\sigma}$.

Now we compute the evolution equations for $\mc{G}_r$ and $\mc{H}_r$. First, since 
$\pdt d\mu = -g^{ij}P_{ij} d\mu$, taking into account the time-dependency of $d\mu$ and the norms $|\cdot|$,
it follows from \eqref{eq:gassumptions} and \eqref{eq:exteqxy} that
 \begin{align}\nonumber
     \mc{G}_r^{\prime}(t) &\leq \int_{M}\left(N\theta|X|^2 + 2\left\langle\pd{X}{t}, X\right\rangle
	  - \pd{\eta}{t}|X|^2\right)\phi_r e^{-\eta}\,d\mu\\
   \begin{split}\nonumber
 	&\leq \frac{N(r^2+1)}{t^{\sigma}}\mc{G}_r 
	+ \int_{M}\left(2\left\langle \Delta X +\operatorname{div}U, X\right\rangle - \pd{\eta}{t}|X|^2\right)\phi_re^{-\eta}
 \,d\mu \\
	   &\phantom{\leq}\quad + \int_{M} \left(2N_1\theta(|X|^2 + |Y|^2) + a|\nabla X|^2\right) \phi_re^{-\eta}\,d\mu
   \end{split}\\
  \begin{split}\label{eq:geq1}
 	&\leq a\mc{K}_r+ \frac{N(r^2+1)}{t^{\sigma}}\left(\mc{G}_r + \mc{H}_r\right)  -\int_M\pd{\eta}{t}|X|^2\phi_re^{-\eta}\,d\mu \\
       &\phantom{\leq}\quad+2\int_{M}\left\langle \Delta X +\operatorname{div}U, X\right\rangle \phi_re^{-\eta}\,d\mu
   \end{split}
   \end{align}
on $M\times (0, T_0]$.
Integrating by parts in the last term in \eqref{eq:geq1}, we find that
\begin{align*}
& 2\int_{M}\left\langle \Delta X +\operatorname{div}U, X\right\rangle \phi_re^{-\eta}\,d\mu\\
 &\quad\leq -2\mc{K}_r + 2\int_{M}\bigg(|\nabla X||U|\phi_r + (|\nabla X||X| + |U||X|)(|\nabla\eta|\phi_r +|\nabla\phi_r|)\bigg)e^{-\eta}\,d\mu,
\end{align*}
and, where $\phi_r > 0$, we can estimate
\begin{align*}\begin{split}
 &2|\nabla X|(|U|\phi_r + |X||\nabla\eta|\phi_r + |X||\nabla\phi_r|)\\
  &\qquad \leq 2(1-a)|\nabla X|^2\phi_r 
  + \frac{3}{2(1-a)}|U|^2\phi_r + \frac{3}{2(1-a)}\left(|\nabla\eta|^2\phi_r+ \frac{|\nabla\phi_r|^2}{\phi_r}\right)|X|^2,
\end{split}
\end{align*}
and
\[
 2|U||X|(|\nabla\eta|\phi_r +|\nabla\phi_r|)\leq 2|U|^2\phi_r + \left(|\nabla\eta|^2\phi_r+\frac{|\nabla\phi_r|^2}{\phi_r}\right)|X|^2.
\]
So, using \eqref{eq:uest} and \eqref{eq:etaeq}, we have
\begin{align}\nonumber
& \int_{M}\left(2\left\langle \Delta X +\operatorname{div}U, X\right\rangle -\pd{\eta}{t}\right) \phi_re^{-\eta}\,d\mu\\
\begin{split}\nonumber
&\quad\leq -2a\mc{K}_r  + N\int_{M}\left(|U|^2 +\frac{5-2a}{2(1-a)}|\nabla \eta|^2 -\pd{\eta}{t}\right)\phi_{r}e^{-\eta}d\mu\\ 
&\quad\phantom{\leq -2a\mc{K}_r\;} + N\int_{\operatorname{supp}{\phi_r}}|X|^2\frac{|\nabla\phi_r|^2}{\phi_{r}}e^{-\eta}d\mu
\end{split}\\
\label{eq:geq2}
&\quad\leq   -2a\mc{K}_r +\frac{N(r^2 + 1)}{t^{\sigma}}(\mc{G}_{r} + \mc{H}_r)  + \frac{N}{r^{2}}\int_{\operatorname{supp}{|\nabla\phi_r|}}|X|^2e^{-\eta}d\mu,
\end{align}
and thus, combining \eqref{eq:geq1} and \eqref{eq:geq2}, that, for any $0< t \leq T_0$,
\begin{align}\label{eq:geq3}
 \mc{G}_r^{\prime} &\leq -a\mc{K}_r +\frac{N(r^2+1)}{t^{\sigma}}\mc{E}_r  + \frac{N}{r^2}\int_{A(x_0, r, r/2)}|X|^2e^{-\eta}d\mu,
\end{align}
where $A(x_0, r, r/2) \dfn B_{g(0)}(x_0, r)\setminus B_{g(0)}(x_0, r/2)$. 

Now we examine the last term in \eqref{eq:geq3}. Since the metric $g(t)$ is uniformly equivalent to $g_0$, we have 
\[
\operatorname{vol}_{g(t)}(A(x_0, r, r/2))\leq \operatorname{vol}_{g(t)}(B_{g_0}(x_0, r)) \leq \gamma^{n/2}A_0e^{A_0r^2},
\]
by the volume growth assumption on $g_0$. Since we also assume that the integrand $|X|^2$ satisfies the 
similar growth bound \eqref{eq:xygrowth}, by choosing $T_0^{\prime} = T_0^{\prime}(n,\gamma, A_0, A_1, B)$ sufficiently small,
we can arrange that
\begin{align*}
 \int_{A(x_0, r, r/2)}|X|^2e^{-\eta}d\mu \leq
 \frac{N}{t^{\sigma}}e^{-\frac{\epsilon r^2}{T_0}},
\end{align*}
for some $\epsilon = \epsilon(A_0, A_1, B) > 0$, provided  $T_0 \leq T_0^{\prime}$.
So we have 
\begin{equation}\label{eq:geq4}
 \mc{G}_r^{\prime}(t) \leq \frac{N(r^2+1)}{t^{\sigma}}(\mc{G}_r + \mc{H}_r) -a\mc{K}_r + \frac{N}{t^{\sigma}r^2}e^{-\frac{\epsilon r^2}{T_0}}
\end{equation}
for any $r > 0$ and $t\in (0, T_0]$, if $T_0 \leq T_0^{\prime}$.

Similarly, by \eqref{eq:exteqxy}, we compute (using here only that $\pd{\eta}{t} \geq 0$) that
\begin{align}
  \nonumber
    \mc{H}_r^{\prime}(t)
      &\leq \int_{M}\left(N\theta|Y|^2 + 2\left\langle\pd{Y}{t}, Y\right\rangle - \pd{\eta}{t}|X|^2\right)\phi_re^{-\eta}\,d\mu\\
  \nonumber
       &\leq \frac{N(r^2+1)}{t^{\sigma}}\mc{H}_r + \int_{M}\left(2N_1\theta(|X|^2 + |Y|^2) + a|\nabla X|^2\right)\phi_re^{-\eta}\,d\mu\\
  \label{eq:heq1}   
    &\leq a\mc{K}_r+ \frac{N(r^2+1)}{t^{\sigma}}(\mc{G}_r +\mc{H}_r).
\end{align}

Combining \eqref{eq:geq4} and \eqref{eq:heq1}, we conclude that, for all $r > 0$ and $t\in (0, T_0]$,  
\[
  \mc{E}_r^{\prime}(t) \leq \frac{N(r^2+1)}{t^{\sigma}}\mc{E}_{r}(t) + \frac{N}{t^{\sigma}r^{2}}e^{-\frac{\epsilon r^2}{T_0}}.
\]
provided $T_0 \leq T_0^{\prime}$.
It follows then, that for any $0 < t_0 < t \leq T_0 \leq T_0^{\prime}$, we have
\begin{align*}
  &e^{-Q(r)t^{1-\sigma}}\mc{E}_r(t) - e^{-Q(r)t_0^{1-\sigma}}\mc{E}_r(t_0)
      \leq \frac{e^{-\frac{\epsilon r^2}{T_0}}}{r^2(r^2+1)}\left(e^{-Q(r)t_0^{1-\sigma}} - e^{-Q(r)t^{1-\sigma}}\right),
\end{align*}
where $Q(r) \dfn N(r^2+1)/(1-\sigma)$.

Now, since $X$ and $Y$ tend to zero uniformly on any compact set, we have $\lim_{t_0\searrow 0} \mc{E}_r(t_0) = 0$ for any fixed $r$. 
Therefore, sending $t_0\searrow 0$, we obtain
\begin{align*}
  \mc{E}_r(t)\leq \frac{e^{-\frac{\epsilon r^2}{T_0}}}{r^2}\left(e^{Q(r)t^{1-\sigma}} -  1\right)
	  &\leq \frac{e^{\frac{NT_0^{1-\sigma}}{1-\sigma}}}{r^2}e^{-\left(\frac{\epsilon}{T_0} - \frac{NT_0^{1-\sigma}}{1-\sigma}\right)r^2}.
\end{align*}
If we choose $T_0$ smaller still, say $T_0 \leq \min\{T_0^{\prime}, (\epsilon(1-\sigma)/(2N))^{1/(2-\sigma)}\}$, the above inequality implies
\[
 \mc{E}_r(t) \leq \frac{e^{\frac{NT_0^{1-\sigma}}{1-\sigma}}e^{-\frac{\epsilon r^2}{2T_0}}}{r^2} \leq \frac{N}{r^2}e^{-\frac{\epsilon r^2}{2T_0}}
\]
for all $r > 0$ and $0 < t \leq T_0$. Fixing $t$ in this range and sending $r\to\infty$ then finishes the argument.
\end{proof}

\end{document}